\newif\ifpictures
\picturestrue

\documentclass[11pt]{amsart}
\usepackage{latexsym}
\usepackage{amssymb}
\usepackage{amsmath}
\usepackage{amscd}
\usepackage[usenames,dvipsnames]{color}
\usepackage{graphicx}
\usepackage{float}
\usepackage{mathrsfs}
\usepackage{enumerate}
\usepackage{tikz}
\usepackage{footnote}
\usepackage{stmaryrd}
\usepackage{hyperref}
\usepackage{wasysym}
\makesavenoteenv{tabular}

\numberwithin{equation}{section}

\tikzset{node distance=1.5cm, auto}

\newtheorem{theorem}{Theorem}[section]
\newtheorem{proposition}[theorem]{Proposition}

\newtheorem{corollary}[theorem]{Corollary}

\theoremstyle{definition}
\newtheorem{definition}[theorem]{Definition}

\theoremstyle{remark}
\newtheorem{example}[theorem]{Example}
\newtheorem{remark}[theorem]{Remark}

\def\<{\langle}
\def\>{\rangle}
\def\lb{\llbracket}
\def\rb{\rrbracket}
\def\ssq{\sqsubseteq}

\def\CC{\mathbb{C}}
\def\NN{\mathbb{N}}

\def\RR{\mathbb{R}}
\def\SS{\mathbb{S}}
\def\ZZ{\mathbb{Z}}

\def\A{\mathcal{A}}

\def\L{\mathcal{L}}
\def\M{\mathcal{M}}
\def\O{\mathcal{O}}
\def\N{\mathcal{N}}
\def\T{\mathbb{T}}
\def\Sp{\mathcal{S}}

\def\x{\mathbf{x}}
\def\z{\mathbf{z}}
\def\w{\mathbf{w}}
\def\t{\mathbf{t}}
\def\1{\mathbf{1}}

\def\Conv{\operatorname{Conv}}
\def\Hom{\operatorname{Hom}}

\def\ord{\operatorname{ord}}
\def\rank{\operatorname{rank}}
\def\Re{\operatorname{Re}}
\def\Row{\operatorname{Row}}
\def\Span{\operatorname{Span}}

\def\endexa{\hfill$\hexagon$}

\newcommand{\struc}[1]{{\color{blue} #1}}

\newcommand{\alp}{\alpha}

\DeclareMathOperator{\eq}{eq}

\title{The lattice of amoebas}
\date{\today}
\author{Jens Forsg{\aa}rd}
\thanks{Research by the first author was supported in part by the SNF project 159240.}
\author{Timo de Wolff}
\address{Jens Forsg{\aa}rd, Universit\'{e} de Gen\`{e}ve, Math\'{e}matiques, Villa Battelle, 1227 Carouge, Suisse\medskip}
\email{Jens.Forsgaard@unige.ch}

\address{Timo de Wolff, Technische Universit\"at Berlin, Institut f\"ur Mathematik, Stra{\ss}e des 17.~Juni 136, 10623 Berlin,
 Germany\medskip}

\email{dewolff@math.tu-berlin.de}


\begin{document}

\begin{abstract}
We study amoebas of exponential sums as functions of the support set $A$.
To any amoeba, we associate a set of approximating sections of amoebas,
which we call \emph{caissons}. We show that a bounded modular lattice of subspaces 
of a certain vector space induces a lattice structure on the set of caissons.
Our results unifies the theories of lopsided amoebas 
and amoebas of exponential sums.
As an application, we show that our theory of caissons yields 
improved certificates for existence of certain components of 
the complement of an amoeba.
\end{abstract}

\maketitle

%%%%%%%%%%%%%%%%%%%%%%%%%%%%%%%
\section{Introduction}

Amoebas were introduced by Gel'fand, Kapranov, and Zelevinsky in \cite{GKZ94} in the context of hypergeometric functions
and discriminants. They have been studied intensively over the last two decades,
primarily as they form a bridge between analytic and algebraic geometry on the one hand,
and tropical geometry and combinatorics on the other hand \cite{deWolff:AmoebaTropicalizationSurvey,Maclagan:Sturmfels,Mikhalkin:Survey}.

Determining the topology of an amoeba is a nontrivial task.
Therefore, methods to approximate amoebas have been developed in, for example, \cite{AKNR13,FMMdW,Purbhoo}.
The main tool in use is the \struc{\emph{lopsided amoeba}}. The term ``lopsided'' 
was introduced by Purbhoo in \cite{Purbhoo}, even though the object appeared earlier in Rullg{\aa}rd's thesis \cite{Rullgard}.
The available methods for approximations of amoebas have not been strong enough to 
tackle the main conjectures remaining in amoeba theory, 
for example the maximally sparse conjecture of \cite{PR04} and Rullg{\aa}rd's 
connectivity conjecture, see Section~\ref{sec:Prel} and \cite[p.~39]{Rullgard}.
This paper is part of an effort to obtain more refined approximations of amoebas. 

In what follows, we discuss two seemingly separate approaches to the approximation problem in the generalized setting of amoebas of exponential sums. First, we investigate how amoebas behave under perturbations of their support. Second, we establish a lattice structure on the set of support sets of amoebas and introduce a new object, the \textit{\struc{$B$-caisson}} of an amoeba, which is a finer approximation than the lopsided amoeba, 
see Definition \ref{def:Caissons}.

We show that these two approaches are essentially the same. Moreover, we demonstrate that the theory of $B$-caissons is in close analogue to the theory of amoebas of exponential sums developed by Favorov and Silipo in 
\cite{Favorov} and \cite{Silipo}. Furthermore, we extend the order map for amoebas to $B$-caissons.
In the end, we obtain a single theory which unifies lopsided amoebas, amoebas of exponential sums, the results by Favorov and Silipo, limits of support sets of amoebas, and functorial properties of amoebas discussed by Rullg{\aa}rd.

One can use $B$-caissons to certify the existence of certain connected components of the complement of an amoeba. Determining the existence of such components in dependence of the coefficients of its defining polynomial is a well-known, notoriously hard problem \cite{PR04,TdW13} with several applications. For example, if we consider a univariate polynomial, then the problem is to understand how the norms of its roots depend on the coefficients. This problem dates back to the late 19th century \cite{Fujiwara, Landau} and saw renewed interest in recent years \cite{AKNR13, Schleicher:Stoll, Theobald:deWolff:Trinomials} due to its impact in numerical methods, complexity theory, and applications.
Using results about exponential sums supported on a barycentric circuit \cite{TdW13}, our theory of $B$-caissons yields new results in this direction, see Corollary \ref{Cor:CircuitCertificate}.

\bigskip

Let us explain our results in more detail. We consider an $(1+n)$-variate exponential sum
\begin{equation}
\label{eqn:ExpSum}
f(\z) \ = \ \sum_{\alpha \in A} c_\alpha \, e^{\<\z, \alpha\>}.
\end{equation}
The set $A$ is called the \struc{\emph{support}} of $f$.
We assume that $f(\z)$ is \struc{\textit{pseudo-homogeneous}}
in the sense that there exists a linear form $\xi \in \Hom(\RR^{1+n}, \RR)$
such that $\<\xi, \alpha\> = 1$ for all $\alpha \in A$.
The \struc{\textit{amoeba}} of an exponential sum, which was introduced by Favorov in \cite{Favorov}, 
is defined by
\[
\struc{\A(f)} \ = \ \overline{ \Re(V)}
\]
where $\struc{V} = V(f) \subset \CC^{1+n}$ is the (analytic) variety of $f$,
and where the closure is taken with respect to the standard topology 
on $\RR^{1+n}$.

Let $\struc{(\CC^*)^A}$ denote the family of
all exponential sums \eqref{eqn:ExpSum} with support $A$.
We often consider its closure $\CC^A$. Let $\struc{N}$ denote the cardinality of $A$. 
We arrange the elements of $A$ as the columns of a $(1+n)\times N$-matrix, which, by slight abuse of notation,
is also denoted by $A$.
The starting point of our analysis is a matrix factorization
\begin{equation}
\label{eqn:Factorization}
A \ = \ T B,
\end{equation}
where $A$ is as above and $B$ is a real, pseudo-homogeneous,
$(1+m)\times N$-matrix.
The matrix $T$ induces an isomorphism of $\CC$-vector spaces 
$\Phi\colon \CC^A \rightarrow \CC^B$,
whose inverse is given by
\[
\struc{\Phi^{-1}(f)(\z)} \ = \ f(\z\, T),
\] 
and an embedding $\struc{\iota}\colon \RR^{n} \rightarrow \RR^{m}$ given by
$\x \mapsto \x\, T$.
We denote the image of $\iota$ by $\struc{H_T}$.

\begin{definition}
\label{def:Caissons}
Let $A = TB$ as above and let $f \in \CC^A$. Then, the \struc{\emph{large}} respectively \struc{\emph{small $B$-caissons}} $\L_B(f)$
and $\A_B(f)$ are defined by
\[
 \struc{\L_B(f)} \ = \ \A(\Phi \circ f)
\quad\text{and}\quad
\struc{\A_B(f)} \ = \ \L_B(f) \cap H_T.
\]
Thus, $\L_B(f) \subset \RR^m$, while we consider $\A_B(f)$ as a subset of $\RR^n$
by taking its inverse image under the map $\iota$.
\endexa
\end{definition}

\begin{example}
\label{ex:TrivialFactorizations}
We can write $A = I\,A$ where $I$ denotes the identity matrix
of size $1+n$. In this case $\Phi$ and $\iota$ act as the identities on
$\CC^A$ and $\RR^{1+n}$ respectively. Thus,
\[
\A_A(f) \ = \ \L_A(f) \ = \ \A(f).
\]
\endexa
\end{example}

\begin{example}
A second factorization of $A$ is given by $A = A I$, where $I$ denotes the identity matrix
of size $N$. Note that $I$ is homogeneous
with respect to the form $\xi = (1, \dots, 1)$. In the algebraic case, $\A_I(f)$ coincides with the 
lopsided amoeba of $f$, see \cite[Section 5]{FMMdW}. 
We refer to $\struc{\A_I(f)}$ as the \struc{\textit{lopsided amoeba}} of $f$
also in the non-algebraic setting.
\endexa
\end{example}

This setup allows the use of two methods in the study of the topology of amoebas:
\begin{enumerate}
\item We can consider $B$-caissons strictly in between the amoeba $\A_A(f)$
and the lopsided amoeba $\A_{I}(f)$. The goal is to find a level where
the structure is refined enough to capture the topology of $\A_A(f)$ but
simple enough to be fully understood.
\item For a fixed set of coefficients, we can consider how the amoeba $\A(f)$ 
depends on the support set $A$ under (small) perturbations.
\end{enumerate}
We show that these two approaches are essentially the same. The limiting
object along a perturbation of the support set $A$ along a
subspace of $\CC^A$ is given by a $B$-caisson of $\A(f)$, where $B$ is determined by the subspace
in question. Conversely, all $B$-caissons of $\A(f)$ arise in this fashion, see Theorems \ref{thm:FundamentalInclusion}, \ref{thm:Silipo}, and \ref{thm:LimitsOfAmoebas}.

%%%%
\subsection{Acknowledgements}
We  thank J.~Maurice Rojas and Laura Felicia Matusevich
for their advice and fruitful discussions. We thank Sascha Timme for his help with creating pictures of amoebas.

The second author was supported by the DFG grant WO 2206/1-1.

%%%%%%%%%%%%%%%%%%%%%%%%%%%%%%%
\section{Preliminaries}
\label{sec:Prel}

In this section we introduce notation and explain key results which are necessary for an understanding of our results. For further background about amoebas and tropical geometry we recommend \cite{Mikhalkin:Survey,Passare:Tsikh:Survey} and \cite{Maclagan:Sturmfels} respectively.

We call the point configuration $A$ \struc{\emph{polynomial}} or
\struc{\emph{algebraic}} if $A \subset \ZZ^{1+n}$. In this case, there
exists a Laurent polynomial $g \in \CC[\w^{\pm 1}]$
such that $f(\z) = g(e^\z)$. The amoeba of the polynomial 
$g$ was defined in \cite{GKZ94} as the image of the
algebraic variety $V(g)$ under the logarithmic absolute value map.
That is, the amoeba of $g$ coincides with the amoeba of $f$ defined above.

Let  $\ZZ[A] \subset \RR^{1+n}$ denote the abelian group generated by the elements of $A$ We let $\struc{\N}$  denote the \struc{\emph{Newton polytope}} of $f$.
That is, $\N = \Conv(A) \subset \RR[A]$, where $\struc{\RR[A]}$ denotes the $\RR$-vector space generated by the elements of $A$.

We denote the $j$th element of $A$ by \struc{$\alpha(j)$} when necessary.
There are two ranks associated to $A$ which are important to us.
First, let $\struc{r(A)}$ denote the rank of the matrix $A$.
Second, let $\struc{\rho(A)}$ denote the rank
of the lattice (abelian group) $\ZZ[A]$. Since $r(A)$ equals the dimension
of $\RR[A] \simeq \RR \otimes \ZZ[A]$,
it holds that $\rho(A) \geq r(A)$. This inequality can be strict.
In particular, given a factorization as in \eqref{eqn:Factorization},
the induced group homomorphism $\struc{T} \colon \ZZ[B] \rightarrow \ZZ[A]$
can be an isomorphism even if the induced linear
transformation $\struc{T} \colon \RR[B] \rightarrow \RR[A]$ is not.
Here, we denote both maps by $T$ with slight abuse of notation since both maps are given by matrix multiplication by $T$.

\begin{remark}
\label{rem:Dehomogenize}
In examples, it is more convenient to reduce the number of variables
by dehomogenizing the exponential sum $f$. There is a standard procedure.
After a change of variables, we can assume that $\xi = (1,0 ,\dots,0 )$.
This implies that the top row of $A$ is the all ones vector. In this case,
$f(\z) = e^{z_0}g(z_1, \dots, z_n)$ for an $n$-variate exponential sum $g$.
\end{remark}

The order map for amoebas of polynomials, which was introduced in \cite{FPT00}, was extend to exponential
sums by Favorov in \cite{Favorov}, where the \emph{\struc{Ronkin function}}
for polynomial amoebas was generalized by \emph{Jessen function}.
In \cite[Section 3.2]{Silipo}, Silipo noted that the Ronkin function
for exponential sums can be defined as follows.
Let $\struc{S} = \Hom_\ZZ\big(\ZZ[A], \,\SS^1\big)$
be the group of $\SS^1$-characters of $\ZZ[A]$,
which is homeomorphic to $(S^1)^{r(A)}$. 
The {Ronkin function} of $f$ is given by
\[
\struc{N_f(\x)} \ = \ \int_S \log|f_\chi(\x)| \,d\eta(\chi),
\]
where $\struc{\eta}$ denotes the \textit{\struc{Haar measure}} on $S$, and $\struc{f_\chi(\x)}$ is the perturbation of $f$ by $\chi$,
\[
\struc{f_\chi(\x)} \ = \ \sum_{\alpha \in A} c_\alpha \, \chi(\alpha)\,e^{\<\x, \alpha\>}.
\]
The gradient of the Ronkin function is constant on the complement of the amoeba $\A(f)$, see for example \cite{PR04}.
Thus, the Ronkin function induces an (injective) map
\[
\struc{\ord} \colon \pi_0(\RR^n \setminus \A) \rightarrow \ZZ[A]\cap \N
\]
called the \struc{\emph{order map}} of the amoeba $\A(f)$, see \cite{FPT00, PR04}.

Let $A$ be fixed, and let $\alpha \in \ZZ[A] \cap \N$. For an exponential sum $f$,
we denote by $\struc{E_\alpha(f)}$ the connected component of $\RR^n \setminus \A(f)$
whose order is $\alpha$. It is possible that $E_\alpha(f) = \emptyset$.
Given a fixed $f$, we denote by $\struc{\Omega(f)} \subset \ZZ[A]\cap \N$ the set of all $\alpha$
such that $E_\alpha(f) \neq \emptyset$.
Given a fixed $\alpha$, we denote by $\struc{U_\alpha(A)} \subset \CC^A$ the set of of all $f \in \CC^A$ such that
$E_\alpha(f) \neq \emptyset$.
Rullg{\aa}rd showed in \cite[Theorem~10]{Rullgard} that, in the algebraic case, the sets $U_\alpha(A)$ are open and semi-algebraic, but left as an open problem whether or not the sets $U_\alpha(A)$ are connected, which we call \emph{Rullg{\aa}rd's connectivity conjecture}.
Silipo extended Rullg{\aa}rd's results to the case of exponential sums, 
see \cite[Theorem~2.6]{Silipo} and Theorem~\ref{thm:Silipo}.

\begin{example}
Consider the factorization $A = TB$ given by
\[
\begin{bmatrix} 1 & 1 & 1 \\ 0 & 1 & \pi \end{bmatrix}
\ = \
\begin{bmatrix} 1 & 0 & 0 \\ 0 & 1 & \pi \end{bmatrix}
\begin{bmatrix} 1 & 1 & 1 \\ 0 & 1 & 0 \\ 0 & 0 & 1\end{bmatrix}.
\]
On the one hand, we have that $\rho(A) = \rho(B) = 3$, so $\ZZ[A] \simeq \ZZ[B] = \ZZ^3$ where the
isomorphism is given by the matrix $T$. On the other hand, we have that $r(A) = 2$ while $r(B) = 3$.
That is, $\RR[A] \simeq \RR^2$ and $\RR[B] \simeq \RR^3$.

To describe the sets $\ZZ[A]\cap \N(A)$ and $\ZZ[B]\cap \N(B)$,
we dehomogenize as in Remark~\ref{rem:Dehomogenize},
(we proceed in an analogous manner in all later examples).
The Newton polytope $\N(B)$ is the standard simplex in $\RR^2$. 
The set $\ZZ[B]\cap \N(B)$ contains three points: the three vertices of the simplex.
The Newton polytope $\N(A)$ is the interval $[0,\pi]$.
The set $\ZZ[A]\cap \N(A)$ is infinite. It follows from \cite{Silipo}, however,
that the image of the order map is still finite, see Section \ref{sec:Relationships}.
\endexa
\end{example}

The \emph{spine} $\Sp$ of an algebraic amoeba was introduced in \cite{PR04}. While it is possible to extend
the concept of spines to amoebas of exponential sums, it suffices for our purposes to consider spines
of polynomial amoebas. The spine $\Sp$ is a tropical variety which is a strong deformation
retract of the (algebraic) amoeba \cite[Theorem 1]{PR04}. In particular, there is an bijection
$\pi_0(\RR^n \setminus \A) \rightarrow \pi_0(\RR^n \setminus \Sp)$,
implying that the order map is well-defined also for the spine $\Sp$.
The spine, as any tropical variety, has a dual triangulation of
the Newton polytope $\N$, see \cite[Definition 2]{PR04}.

%%%%%%%%%%%%%%%%%%%%%%%%%%%%%%%
\section{The Poset Lattice}
\label{sec:Lattice}

In this section we describe the lattice structure induced by the relationship \eqref{eqn:Factorization}. This lattice structure is the theoretical framework within which we establish a lattice structure of amoebas.

Let us denote by $\struc{M_N = M_N(\RR)}$ the set of real matrices $A$ with $N$ columns and any number of rows, such that there exists a linear form $\xi$ with $\xi A = \1$, where $\1 = (1, \dots, 1)$.
That is, the vector $\1$ belongs to the real row span $\struc{\Row(A)}$ of $A$.

\begin{definition}
We define an equivalence relation $\struc{\sim}$ on $M_N$ by
\begin{equation}
 \struc{A \sim B} \ \Leftrightarrow \ \Row(A) = \Row(B).
 \label{Equ:DefEquivalenceRelation}
\end{equation}
We denote the quotient $M_N/\sim$ by $\struc{\M_N}$,
and for $A \in M_N$ we denote the corresponding equivalence class 
by $\struc{\llbracket A \rrbracket} \in \M_N$.
We write $\struc{\Row(\lb A \rb)}$ for the row space defined by any representative
of $\lb A \rb$.
\endexa
\end{definition}

\begin{remark}
\label{rem:AllSubspacesAppear}
Let $L \subset \RR^N$ be a subspace of dimension $d$ containing $\1$. 
By choosing a basis of $L$ we obtain an $d \times N$-matrix
$A$ with $\Row(A) = L$. We make two remarks.
First, $L$ is realized as the row space of some $\lb A\rb \in \M_N$.
Second, each equivalence class $\llbracket A \rrbracket \in \M_N$ has 
a representative $A$ which is an $(1+n) \times N$-matrix where $1+n \leq N$.
\end{remark}

\begin{definition}
\label{def:PartialOrder}
We define a partial order on $\M_N(\RR)$ by
\[
  \struc{\llbracket B \rrbracket \sqsubseteq \llbracket A \rrbracket} \ \Leftrightarrow \
  \Row(\lb A\rb) \subseteq \Row(\lb B\rb).
  \]
We explain the reversal of the inequalities in Section \ref{ssec:LatticeMorphisms}.
\endexa
\end{definition}

Notice that $A \sim B$ if and only if there are matrices $S$ and $T$ such that $A = TB$
and $B = SA$.
And, $\lb B \rb \ssq \lb A \rb$ if and only if for any choice of representatives $A$ and $B$
there is a matrix $T$ with $A = TB$.

\begin{remark}
The rank function $r(A)$ is constant on the classes $\llbracket A \rrbracket$, 
as it equals the dimension of $\Row(A)$.
Hence, it induces a function $\struc{r}\colon \M_N \rightarrow \NN$.
To adjust for the reversal of the inequalities in Definition~\ref{def:PartialOrder},
we define $\struc{\hat r(\lb A\rb)} = N-r(\lb A\rb)$. Notice that $\hat r(\lb A \rb)$
is the rank of the orthogonal complement of $\Row(A)$, and
since $\Span(\1) \subset \Row(A)$ we have that $\hat r(\lb A \rb) \in \{0, \dots, N-1\}$.
\end{remark}

\begin{theorem}
The space $(\M_N, \ssq)$ with the rank function $\hat r$
is isomorphic to the bounded modular lattice of all subspaces of\/ $\RR^{N-1}$.
\end{theorem}

\begin{proof}
The space of all subspaces of $\RR^{N-1}$ is a bounded modular lattice with respect to inclusion, where
the grading is given by the vector space dimension.
The isomorphism is given by the map $\lb A \rb \mapsto \Row \lb A \rb^\perp$. 
It is injective by definition, and surjective by Remark~\ref{rem:AllSubspacesAppear}.
It is straightforward to check the remaining conditions of a lattice.
\end{proof}

Consider an exponential sum $f$. 
The matrix $A$ is obtained by choosing an ordering of the elements of the set $A$.
It can happen that the matrices one obtains from distinct choices of orderings define distinct 
equivalence classes in $\M_N$.
One could enlarge the equivalence classes by allowing also permutations
of columns. This would not introduce any additional difficulty, but yields a less clear notation and no further insights.
Thus, we have chosen not to take this approach. 

\begin{proposition}
The function $\rho\colon \M_N\rightarrow \ZZ$ given by $\lb A  \rb \mapsto \rank(\ZZ[A])$ is well-defined and decreasing.
\end{proposition}

\begin{proof}
Let $A = TB$ for some $T$. It follows that $\ZZ[A]$ is the image of $\ZZ[B]$
under the group homomorphism induced by $T$. In particular, any set of generators of $\ZZ[B]$ projects by $T$ onto
a set of generators of $\ZZ[A]$, so $\rho(B) \geq \rho(A)$.
\end{proof}

\begin{definition}
\label{def:MinimalLift}
We call a matrix $A$ with $\rho(A) = r(A)$ a \struc{\emph{rational point configuration}}, and the corresponding class $\llbracket A \rrbracket$ a \struc{\emph{rational}} class.
Let $\llbracket A \rrbracket \in \M_N$. We call a rational class $\llbracket B \rrbracket$
satisfying $\llbracket B \rrbracket \sqsubseteq \llbracket A \rrbracket$
a \struc{\emph{rational lift}} of $\lb A \rb$.
If, in addition, $\rho(\llbracket B \rrbracket) = \rho(\llbracket A \rrbracket)$ 
then we call $\lb B \rb$ a \struc{\emph{minimal rational lift}} of $\lb A\rb$.
\endexa
\end{definition}

The following universal property of minimal rational lifts shows that any rational lift of $\lb A \rb$
is a composition of the unique minimal rational lift and and integer (algebraic)
transformation. This decomposition is one our main tools to understand the relationship
between an amoeba and its caissons.

\begin{theorem}
\label{thm:MinimalRationalLifts} 
Each $\lb A \rb \in \M_N$ has a unique minimal rational lift, denoted as \/ $\widehat{\lb A \rb}$,
fulfilling the universal property that every rational lift of\/ $\lb A \rb$ is a lift of\/ $\widehat{\lb A \rb}$.
\end{theorem}

\begin{proof}
An element $\lb B \rb \in M_N$ is rational if and only if $\Row(B)$
is \emph{rational}, i.e., it has a basis consisting of rational vectors. Given two rational 
subspaces $\Row(B_1)$ and $\Row(B_2)$ of $\RR^N$ containing $\Row(A)$,
then $\Row(B_1) \cap \Row(B_2)$ is also a rational subspace containing $\Row(A)$.
Hence, if $\Row(B_1)$ is minimal, then $\Row(B_1) = \Row(B_1) \cap \Row(B_2)$,
implying both claims of the theorem.
\end{proof}

%%%%%%%%%%%%%%%%%%%%%%%%%%%%%%%
\section{The Relationship between an Amoeba and its Caissons}
\label{sec:Relationships}

In this section, we apply the results about the poset lattice from Section~\ref{sec:Lattice} to understand the
relationship between an amoeba and its associated caissons. 
We obtain a lattice structure on the set of amoebas, 
and an associated lattice structure on the set of sets of orders.

\begin{theorem}
\label{thm:FundamentalInclusion}
Let $\llbracket B \rrbracket \sqsubseteq \llbracket A \rrbracket$.
Then, for any $f \in \CC^A$, it holds that $\A \subset \A_B$ and $\iota(\A) \subset \L_B$.
\label{Thm:AmoebaContainment}
\end{theorem}

\begin{proof}
It is sufficient to show the second inclusion. Let $g = \Phi\circ f \in \CC^B$.
Let $\x \in \A$, so that $\x = \Re(\z)$ for some
$\z$ with $f(\z) = 0$. Then, $g(\z\,T) = 0$, implying that $\Re(\z\,T) \in \L_B$.
Therefore,
$\x\, T = \Re(\z)T = \Re(\z\,T) \in \L_B$.
where the second equality holds since $T$ is a real matrix.
\end{proof}

In Example~\ref{ex:TrivialFactorizations}, the fundamental inclusion
$\A \subset \A_B$ was an equality. It is an interesting
problem to determine when such an equality holds.
We include Silipo's Theorem \cite[Theorem~2.6]{Silipo} here,
as it is the fundamental result on amoebas of exponential sums.
We formulate the theorem in the language of Section \ref{sec:Lattice}.

\begin{theorem}[{Silipo}]
\label{thm:Silipo}
Let $\lb B \rb$ be the minimal rational lift of\/ $\lb A \rb$,
and let $f \in \CC^A$. Then, $\A(f) = \A_{B}(f)$. \qed
\end{theorem}

To describe the relationship between $\A$ and $\A_B$ in more detail,
we study relationship between the associated order maps.
Indeed, Theorem~\ref{thm:FundamentalInclusion} implies that there exists a well-defined map
\[
\pi_0(\RR^n \setminus \A_B) \ \rightarrow \ \pi_0(\RR^n \setminus \A),
\]
given by inclusion of subsets of $\RR^n$.

\begin{theorem}
\label{thm:Orders}
Let $\lb B \rb \ssq \lb A \rb$, and let $f \in \CC^A$ with $g = \Phi \circ f \in \CC^B$.
If $E_\beta(g)$ meets $H_T$, so that $E_\beta(g) \cap H_T \subset E$
for some $E \in \pi_0(\RR^n \setminus \A)$,
then $\ord_f(E) = T \beta$.
\end{theorem}

\begin{proof}
In the case that $\lb A \rb$ and $\lb B \rb$ are both rational, this follows from
Rullg{\aa}rd's trick to compute orders from \cite[p.~20]{Rullgard}.
In the case that $\lb B\rb$ is a minimal rational lift of $\lb A \rb$, this follows from
Silipo's analysis in \cite[Section 3.2]{Silipo} (see in particular the proof of \cite[Lemma 3.6]{Silipo}). The general case can always be reduced to
a composition of these two cases, using Theorem~\ref{thm:MinimalRationalLifts},
as follows. Let $\widehat A$ and $\widehat B$ be the minimal rational lifts of $A$ and $B$.
By universality of of $\widehat A$, we obtain a commutative diagram of abelian groups
\begin{center}
\begin{tikzpicture}
  \node (A1) at (0,0)  {$\ZZ[\widehat B]$};
  \node (A2) at (2,0)  {$\ZZ[\widehat A]$};
  \node (B1) at (0,-2)  {$\ZZ[B]$};
  \node (B2) at (2,-2)  {$\ZZ[A]$};
  \draw[->] (A1) to node {$\widehat T$} (A2);
  \draw[->] (A1) to node {$S_B$} (B1);
  \draw[->] (A2) to node {$S_A$} (B2);
  \draw[->] (B1) to node {$T$} (B2);
\end{tikzpicture}
\end{center}
We find that $\beta = S_B(\gamma)$
 for some $\gamma \in \ZZ[\widehat B]$ by Silipo's results mentioned above, and by using both
Rullg{\aa}rd's and Silipo's results we find that the order of $E$ is equal to
\[
\ord_f(E) \ = \ (S_A\circ \widehat T) (\gamma) \ = \ (T \circ S_B)\, (\gamma) \ = \ T(\beta).\qedhere
\]
\end{proof}

\begin{theorem}
\label{thm:Injective}
Let $\lb B\rb \ssq \lb A \rb$. Then, we obtain an injective map $\pi_0(\RR^n \setminus \A_B) \rightarrow \pi_0(\RR^n \setminus \A)$ given by inclusion.
\end{theorem}

\begin{proof}
By Theorems~\ref{thm:MinimalRationalLifts} and \ref{thm:Silipo} it suffices to show that
$\pi_0\big((\RR^n \setminus \A_B)\cap H\big) \rightarrow \pi_0\big((\RR^n \setminus \A)\cap H\big)$
is injective whenever $A$ and $B$ are rational and $H$ is an arbitrary affine subspace of $\RR^n$.
By convexity of the components of the complement of an amoeba, it suffices to show that 
$\pi_0(\RR^n \setminus \A_B) \rightarrow \pi_0(\RR^n \setminus \A)$ is injective whenever $A$
and $B$ are rational.

This last statement follows from the existence of spines of amoebas.
Let $T$ be such that $A = TB$, and consider the caisson $\L_B \subset \RR^m$.
By definition of $\A_B$ and $\L_B$, it suffices to show that $T$ is injective when restricted to the set of orders of $\pi_0\big( (\RR^m \cap \L_B)\cap H_T\big)$.
It is equivalent to show the same statement when $\L_B$ is replaced by its spine $\Sp_B$.
 Since the connected components
of the complement of $\Sp_B$ are open and convex there is no restriction in assuming that 
$H_T$ is rational and in general position with respect to the spine $\Sp_B$
(i.e., its intersection with any cell of $\Sp_B$ is an affine space of expected dimension).
Then, $\Sp = \Sp_B \cap H_T$ is a tropical variety in $\RR^n$.
Hence, the vertices of the dual triangulation to $\Sp$, 
which are all distinct, correspond bijectively to the set of projections 
$T \beta$ where $\beta$ ranges over the set $\pi_0\big( (\RR^m \cap \Sp_B)\cap H_T\big)$. 
\end{proof}

\begin{definition}
Let $\lb B \rb \ssq \lb A\rb$. We define the \struc{\emph{order map of the $B$-caisson}}
$\A_B$ to be the composition of the map $\pi_0(\RR^n \setminus \A_B) \hookrightarrow \pi_0(\RR^n \setminus \A)$ from Theorem~\ref{thm:Injective} and the order map of the amoeba $\A$.
We denote by $\struc{\Omega_B(f)} \subset \ZZ[A]\cap \N$ the image of the order map
for the $B$-caisson $\A_B(f)$.
\endexa
\end{definition}

\begin{remark}
It is clear that $\Omega_B(f)$ only depends on the class $\lb B \rb$. Indeed, if $B$ and $B'$
represent the same class, then we obtain injections
\[
\pi_0(\RR^n \setminus \A_B) \ \hookrightarrow \
\pi_0(\RR^n \setminus \A_{B'}) \ \hookrightarrow \ \pi_0(\RR^n \setminus \A_B).
\]
Since moreover both sets are finite, see for example \cite[(1.15)]{Silipo}, they coincide.
\end{remark}

%%%%%%
\subsection{An Interpretation as Lattice Morphisms}
\label{ssec:LatticeMorphisms}

Let us, in this subsection, fix the point configuration $A$.
For any $f \in \CC^A$, we have $\Omega(f) \subset \ZZ[A] \cap \N(A)$.
We could, in what follows, take our starting point in the set $\ZZ[A] \cap \N(A)$.
However, this set is larger than necessary, as, for example, it might be infinite even though
$\Omega(f)$ is always finite.
Following Silipo, let us consider 
the following construction. Let $\lb B \rb$ be a minimal rational lift of $\lb A \rb$,
with the induced isomorphism $T \colon \ZZ[B] \rightarrow \ZZ[A]$.
Since $\lb B \rb$ is rational, the set $\ZZ[B] \cap \N(B)$ is finite.
Let $\omega(A) \subset \ZZ[A] \cap \N(A)$ denote the set $T(\ZZ[B] \cap \N(B))$, 
which is finite as $T$ is an isomorphism of abelian groups. 
Let $\O(A)$ denote the bounded modular lattice of all subsets of $\omega(A)$.
Then, Theorem~\ref{thm:Injective} can be restated as follows.

\begin{theorem}
Each $f \in \CC^A$ induces a morphism
$\zeta(f) \colon \M_N(A) \rightarrow \O(A)$ of bounded modular lattices, 
where $\M_N(A) \subset \M_N$
denotes the lattice of all lifts of the class\/ $\lb A \rb$. \qed
\end{theorem}

%%%%%%%%%%%%%%%%%%%%%%%%%%%%%%%
\section{Continuity of Amoebas as Functions of the Support}

Amoebas have been considered in various contexts with respect to their configuration spaces $\CC^A$ for a fixed support set $A$. It is a natural question to ask, how amoebas depend on the choice of support set. Particularly, considering amoebas of exponential sums one would like to understand how an amoebas behave under a limit process of support sets. In this section we show that limits of amoebas with respect to their support sets are caissons.

Let us consider the deformation of the exponential sum $f(\mathbf{z})$ as in \eqref{eqn:ExpSum} given by
\[
f_\lambda(\z,\t) \ = \ \sum_{\alpha \in A} c_\alpha e^{\<\alpha, \z\> + \lambda\<\kappa_\alpha, \t \>}
\]
where $\t = (t_1, \dots, t_k) \in \CC^k$ are additional variables, the vectors $\kappa_\alpha \in \RR^k$, and $\lambda$ is a real parameter.
We do not exclude the case that some $\kappa_\alpha = 0$.
We wish to consider the limit 
\[
\lim_{\lambda \rightarrow 0} \A(f_\lambda),
\]
which, by definition, is the (closure of) the set of all limit points of sequences $(\z_k,\t_k) \in \A(f_{\lambda_k})$
where $\lambda_k \rightarrow 0 $ as $k\rightarrow \infty$
and $\lambda_k \neq 0$ for all $k$.
It is a consequence of our proof of the forthcoming
Theorem~\ref{thm:LimitsOfAmoebas} that this limit does not depend on the choice of the sequence 
$\{\lambda_k\}_{k=0}^\infty$.

Let us construct a matrix $B$ as the $(n+k)\times N$ matrix 
\[
B \ = \ \begin{bmatrix}\alpha_1 & \alpha_2 &  \cdots & \alpha_N \\ \kappa_1 & \kappa_2 & \cdots & \kappa_N\end{bmatrix},
\]
and note that $\lb B \rb$ is a lift of $\lb A \rb$ as $A = T B$ where, in block form, $T = (I, 0)$.

\begin{theorem}
\label{thm:LimitsOfAmoebas} We have that
\[
\lim_{\lambda\rightarrow 0} \A(f_\lambda) \ = \ \A_B(f) \times \RR^k.
\]
\end{theorem}

Before we turn to the proof of Theorem~\ref{thm:LimitsOfAmoebas},
let us introduce some additional notation.
Consider the auxiliary exponential sum
\[
g(\z, \t) \ = \ \sum_{i=1}^N c_\alpha e^{\<\alpha_i, \z\> + \<\kappa_\alpha, \t \>}.
\]
We have that $f_\lambda(\z, \t) = g(\z, \lambda \t)$. By definition, $\A_B(f) = \A(g) \cap H_T$.

\begin{proof}[Proof of Theorem~\ref{thm:LimitsOfAmoebas}]

\noindent
Fix a sequence $\{\lambda_k\}_{k=0}^\infty$ such that $\lambda_k \rightarrow 0$
as $k\rightarrow \infty$. Consider an arbitrary associated convergent sequence
\[
\left\{(\z_k, \t_k)\right\}_{k=0}^\infty \ \rightarrow \ (\z^*, \t^*)
\]
with $\Re(\z_k, \t_k) \in \A(f_{\lambda_k})$ for all $k$. It follows that $\Re(\z_k, \lambda_k \t_k) \in \A(g)$. 
Since the sequence $\{\t_k\}_{k=0}^\infty$ converges, it holds that
\[
\lim_{k\rightarrow \infty} \lambda_k \t_k \ = \ 0.
\]
Thus, we have $\lim_{k\rightarrow \infty} \Re(\z_k, \lambda_k\t_k) \in H_T$. Since $\A(g)$ is closed, we know moreover that the limit is contained in $\A(g)$ and we can conclude that 
\[
\Re(\z^*, 0) \ = \ \lim_{k\rightarrow \infty} \Re(\z_k, \lambda_k\t_k) \in \A(g) \cap H_T \ = \ \A_B(f).
\]
In particular, since $\t^{*} \in \RR^k$, 
\[
\lim_{k \rightarrow \infty} \A(f_{\lambda_k}) \ \subset \ \A_B(f) \times \RR^k.
\]

For the converse inclusion, let $\{\lambda_k\}_{k=0}^\infty$ be a sequence such that $\lambda_k \rightarrow 0$, and let $(\z^*, \t^*) \in \A_B(f) \times \RR^k$. 
As $\A_B(f) \times \RR^k$ is a regular closed set (i.e., it is the closure of its interior) we can find a sequence $\{(\z_m, \t_m)\}_{m=0}^\infty$ contained in the interior of 
$\A_B(f) \times \RR^k$ such that 
\[
\lim_{m\rightarrow \infty} (\z_m, \t_m) \ = \ (\z^*, \t^*).
\]
In particular, $\z_m$ is an interior point of $\A_B(f)$ and $\A_B(f) = \A(g)\cap H_T$.
We have for every $\z_m$ that
\[
\lim_{k\rightarrow \infty} (\z_m,  \lambda_k \t_m) \ = \ (\z_m, 0),
\]
and thus, since $\z_m$ is an interior point of $\A_B(f)$, for each $m\in \ZZ$ there exists a 
$K(m) \in \ZZ$ such that $(\z_m,  \lambda_k \t_m) \in \A(g)$ for $k \geq K(m)$.
Choose an increasing sequence $\{k_m\}_{m=0}^\infty$
such that $k_m \geq K(m)$; by taking a subsequence of $\{\lambda_k\}_{k=0}^\infty$ we can assume
that $k_m = m$.
Finally, $(\z_m,  \lambda_m \t_m) \in \A(g)$ is equivalent to $(\z_m, \t_m) \in \A(f_{\lambda_m})$
as $\lambda_m\neq 0$.
It follows that
\[
(\z^*, \t^*) \ =\  \lim_{m\rightarrow \infty} (\z_m, \t_m) \in \lim_{m\rightarrow \infty} \A(f_{\lambda_m}),
\]
which concludes the proof.
\end{proof}

%%%%%%%%%%%%%%%%%%%%%%%%%%%%%%%
\section{New Certificates for the Existence of Components\\ of the Complement of Amoebas}
\label{Sec:NewBounds}

In this section we show that our approach allows to certify the existence of specific components of the complement of certain amoebas in dependence of the coefficients of its defining exponential sum. These certificates improve previous certificates obtained via lopsided amoebas.

The main problem is to find a semi-algebraic description of a set $U_\alpha(A)$.
 If $\alpha$ is a vertex of $\Conv(A)$, then it follows by \cite[Corollary 1.8, p.~196]{GKZ94} that $U_\alpha(A) = \CC^A$.
If $\alpha$ is not a vertex of $\Conv(A)$ a semi-algebraic description of $U_\alpha(A)$ is unknown in all cases except when $A$ is a barycentric circuit. 

\begin{definition}
We call a support $A = \{\alp(0),\ldots,\alp(n),\gamma\}$ of cardinality $N = n+ 2$
a \struc{\textit{barycentric circuit}} if $\alp(0),\ldots,\alp(n)$ are the vertices of an $n$-dimensional simplex and $\gamma$ is the barycenter of the simplex. If an exponential sum is supported on a barycentric circuit, we say that $f$ is a barycentric circuit. Every barycentric circuit $f$ is of the form 
\begin{equation}
f(\z) \ = \ \sum_{j = 0}^n c_j e^{\langle \mathbf{z},\alp(j) \rangle} - c_\gamma e^{\langle \mathbf{z}, \gamma\rangle} \ 
\end{equation}
For a barycentric circuit $f$ we define $\struc{\eq(f)} \in \RR^n$ as the real part of the unique point where all terms $c_j e^{\langle \mathbf{z},\alp(j) \rangle}$ for $j = 0,\ldots,n$ are \struc{\textit{in equilibrium}}, that is, where they attain the same absolute value.
\endexa
\end{definition}

In the case the $A$ is algebraic we abuse notation and write
\begin{equation}
f(\w) \ = \ \sum_{j = 0}^n c_j \mathbf{w}^{\alp(j)} - c_\gamma \mathbf{w}^\gamma \label{Equ:CircuitPolynomial}.
\end{equation}

\begin{theorem}[{\cite[Theorem 6.1]{TdW13}}]
Let $A$ be an algebraic barycentric circuit. Then, the following statements are equivalent.
\begin{enumerate}[i)]
	\item  $f \in U_\gamma(A)$ 
	\item  $c_\gamma$ is not contained in the region 
 \begin{equation}
	 \Bigg\{\prod_{j = 0}^n {|c_j|}^\frac{1}{1+n} \cdot \sum\limits_{j = 0}^n e^{i \cdot (\arg(c_j) + \langle \alp(j) - \beta, \phi \rangle) } : \phi \in \T^n\Bigg\}, \label{Equ:CircuitPolynomialRegionSolid}
	\end{equation}
	where $\struc{\T} = [0,2\pi)$.
	\item  $\eq(f) \in E_\gamma(f)$,
\end{enumerate}
\label{Thm:Theobald:deWolff:BarycentricCircuitCase}
\end{theorem}

The region \eqref{Equ:CircuitPolynomialRegionSolid} contains the origin and is bounded by a hypocycloid, see \cite[Section 6]{TdW13} for further details. 
With the results of this article we obtain the following corollary of Theorem \ref{Thm:Theobald:deWolff:BarycentricCircuitCase}.

\begin{corollary}
\label{cor:Certificates}
Let $B$ be a rational lift of $A$ which is an algebraic barycentric circuit.
Let $f \in \CC^A$ and let $g = \Phi \circ f \in \CC^B$.
If $\eq(g) \in H_T$ and $c_\gamma$ is not contained in the region \eqref{Equ:CircuitPolynomialRegionSolid}, then $E_{T\gamma}(f) \neq \emptyset$.
\label{Cor:CircuitCertificate}
\end{corollary}

\begin{proof}
If $c_\gamma$ is not contained in the region defined in \eqref{Equ:CircuitPolynomialRegionSolid}, then $E_\gamma(g) \neq \emptyset$ by Theorem \ref{Thm:Theobald:deWolff:BarycentricCircuitCase}. Thus, the statement follows from Theorems~\ref{Thm:AmoebaContainment} and \ref{thm:Orders}.
\end{proof}

In what follows we provide two examples demonstrating Corollary~\ref{cor:Certificates}.

\begin{example}
\label{ex:Sec6:Ex1}
Consider the family of univariate polynomials
\[
f(w) =  1 + w^3 + c\, w^4  + \, w^9,
\]
where $c \in \CC$. We associate to $f$ its homogeneous support set $A$ (see Remark~\ref{rem:Dehomogenize})
and rational lift $B$ given by 
\begin{equation}
\label{Equ:Exa1Support}
A \ = \
\begin{bmatrix} 1 & 1 & 1 & 1 \\ 0 & 3 & 4 & 9\end{bmatrix} 
\quad \text{and} \quad
B \ = \
\begin{bmatrix} 1 & 1 & 1 & 1 \\ 0 & 3 & 4 & 9 \\0 & 6 & 2 & 0\end{bmatrix},
\end{equation}
see Figure \ref{Fig:NewBounds1}
Set $g = \Phi \circ f$, so that
\[
g(\w) \ = \ 1 + w_1^3 w_2^6 + c \,w_1^4 w_2^2  +  w_1^9.
\]
For every fixed $c$, let $\struc{a_1,\ldots,a_9} \in \CC$ denote the roots of $f(w)$ ordered with respect to magnitude. The set of norms of these roots is the image of $\A(f)$ under the exponential map. 
We would like to determine for which $c$ we have that $|a_4| = |a_5|$.
The lopsidedness certificate in this example implies that $|a_4| \neq |a_5|$ if 
\[
c \ > \ \min\{ w^{-4} + w^{-1} + w^5 : w \in \RR\} \ = \ 3.
\]  
It follows from Corollary \ref{Cor:CircuitCertificate} that $|a_4| < |a_5|$ if $c \in \CC$ is not contained in the region 
\eqref{Equ:CircuitPolynomialRegionSolid}. Let us write $c = re^{i \theta}$. Then, by applying a Gr\"obner basis computation, the boundary of \eqref{Equ:CircuitPolynomialRegionSolid} is given by the hypocycloid $h(r, \theta) = 0$ 
\begin{equation}
 h(r, \theta) \ = \ -27 + 18 r^2 + r^4 - 8 r^3 \cos(3 \theta).
 \label{Equ:Exa1Hypocycloid}
\end{equation}
In particular, if $h(r,\theta) > 0$, then $|a_4| < |a_5|$. Notice that $r > 3$ implies that $h(r, \theta)  > 0$,
so our certificate is an improvement of the lopsidedness certificate. See Figure~\ref{Fig:NewBounds1}
for a comparison.

For example, if we require that $|c| > 1.5$ then we obtain the following numerical 
intervals in the argument of $c$ which ensures that $|a_4| < |a_5|$:
\[
 \arg(c) \in [-0.25\pi, 0.42\pi] \ \cup \
  [0.91\pi, -0.91\pi] \ \cup \ [-0.42\pi, -0.25\pi].
\]
Similarly, if $|c| > 2.5$ then we obtain the following numerical 
intervals:
\[
\arg(c) \in [0.32\pi, 0.34\pi] \ \cup \ [0.99\pi, -0.99\pi] \ \cup \ [-0.34\pi, -0.32\pi].
\]
\endexa
\end{example}

\begin{figure}
\ifpictures
\includegraphics[width=0.29\linewidth]{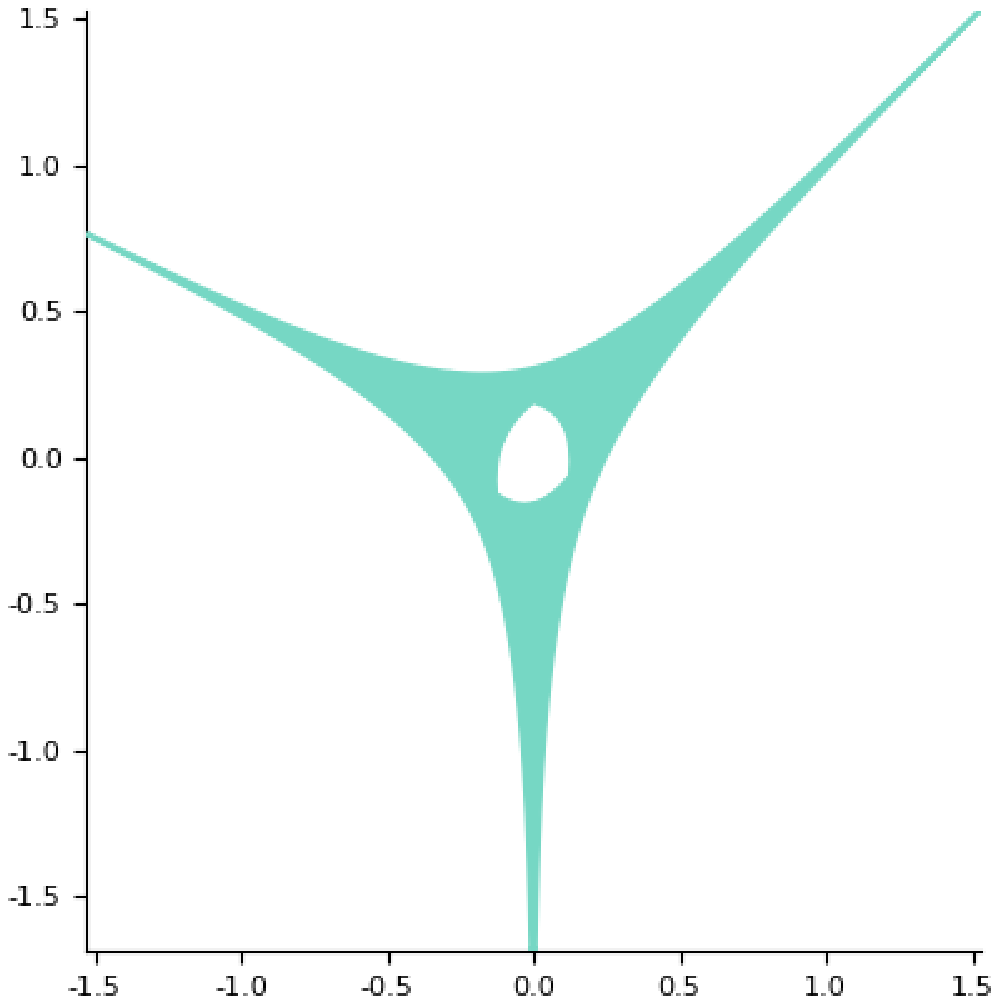} \quad
\includegraphics[width=0.34\linewidth]{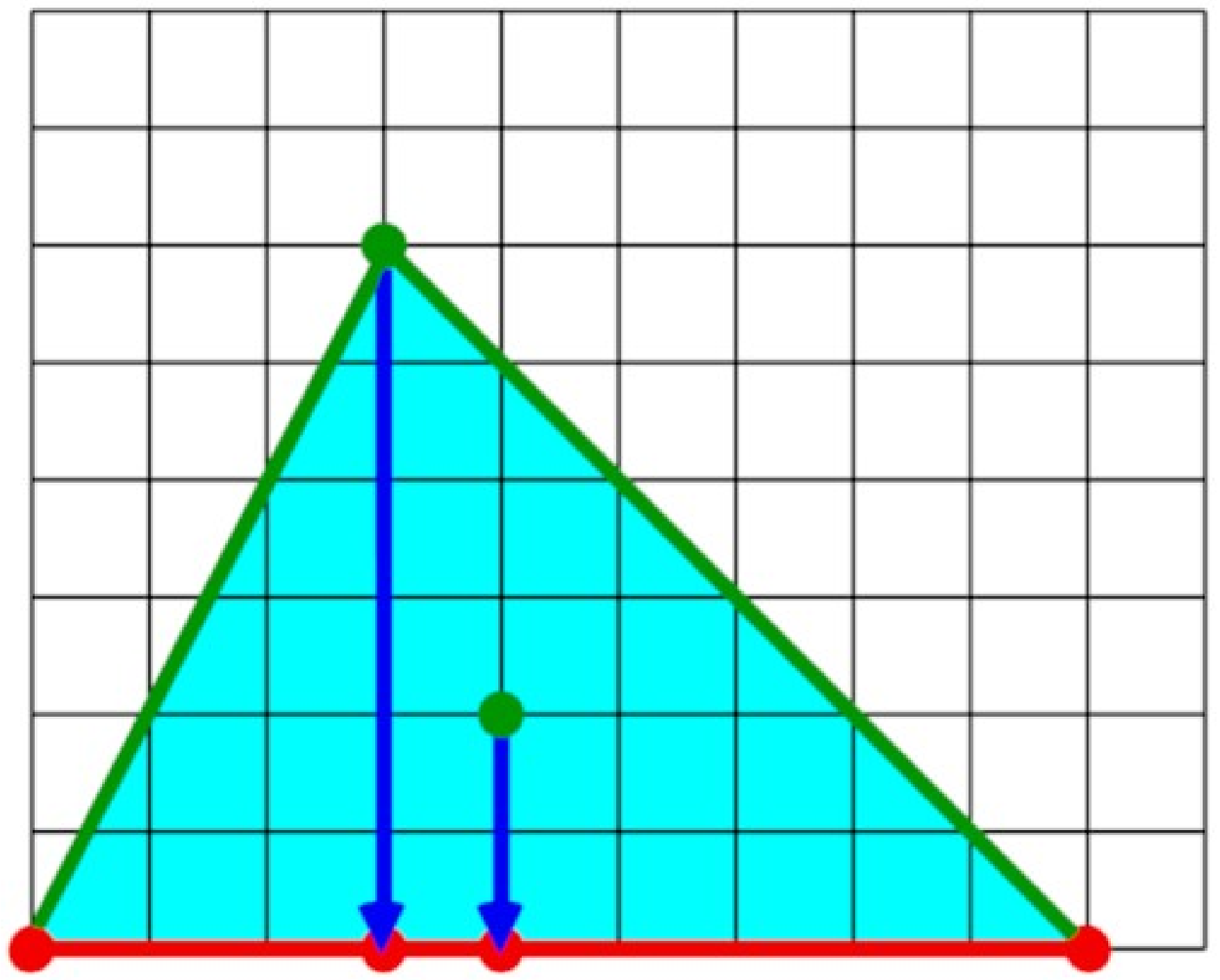} \quad
\includegraphics[width=0.29\linewidth]{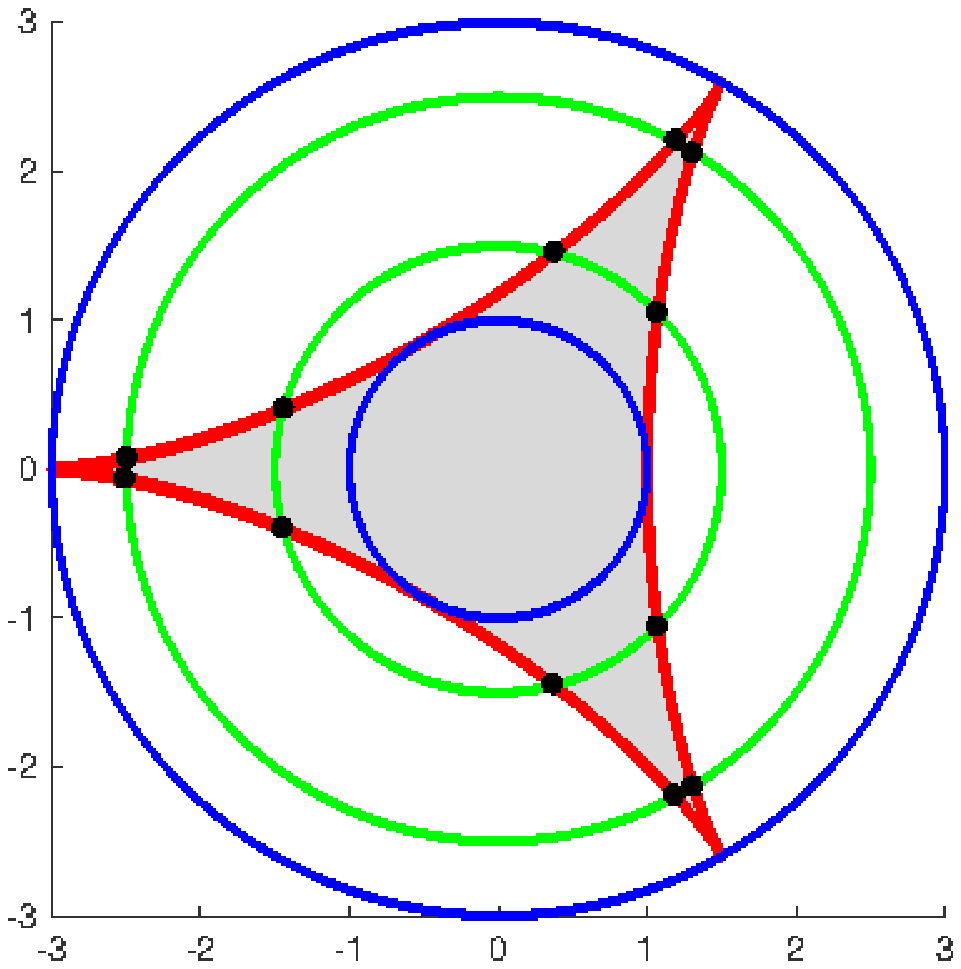}
\fi
\caption{Left, the amoeba of $g(\w)$ from Example~\ref{ex:Sec6:Ex1} for $c = 2.5\, e^{\frac 5 6 \pi i}$. Middle, the Newton polytope of $B$ with its projection onto $A$. 
Right, the region \eqref{Equ:CircuitPolynomialRegionSolid} bounded by the hypocycloid $h(r,\theta) = 0$ 
in red.}
\label{Fig:NewBounds1}
\end{figure}

\begin{example}
\label{ex:Sec6:Ex2}
Consider the family of bivariate polynomials
\[
f(w_1,w_2) \ = \ 1 +w_1^2 w_2^2 + c \,w_1^3 w_2^3+w_1^4 w_2^6 + w_1^6 w_2^4,
\]
where $c \in \CC$. We associate to $f$ its homogeneous support set $A$ (see Remark~\ref{rem:Dehomogenize})
and a rational lift $B$ given by 
\begin{equation}
\label{Equ:Exa2Support}
A \ = \
\begin{bmatrix}  1 & 1 & 1 & 1 & 1 \\ 0 & 2 & 3 & 4 & 6 \\  0 & 2 & 3 & 6 & 4 \end{bmatrix},
\quad \text{and} \quad
B \ = \
\begin{bmatrix}  1 & 1 & 1 & 1 & 1 \\  0 & 2 & 3 & 4 & 6 \\
 0 & 2 & 3 & 6 & 4 \\ 0 & 4 & 1 & 0 & 0 \end{bmatrix},
\end{equation}
see Figure \ref{Fig:NewBounds2}. Set $g = \Phi \circ f$, so that
\[
g(w_1, w_2,w_3) \ = \ 1 + w_1^2 w_2^2 w_3^4 + c \, w_1^3 w_2^3 w_3+ w_1^4 w_2^6 + w_1^6 w_2^4.
\]
By the lopsidedness criterion we can conclude that $\A(f)$ 
has a bounded component of order $(3,3)$ if $|c| > 4$.
It follows from Corollary \ref{Cor:CircuitCertificate} that $\A(f)$ 
has a bounded component of order $(3,3)$ if $c \in \CC$ is not contained in the region 
\eqref{Equ:CircuitPolynomialRegionSolid}. Again, let us write $c = re^{i \theta}$. Then, the boundary of \eqref{Equ:CircuitPolynomialRegionSolid} is in this case given by the hypocycloid $h(r, \theta) = 0$ where
\begin{equation}
 h(r, \theta) \ = \ -4096 + 768 r^2 + 6 r^4 + r^6 - 54 r^4 \cos(4 \theta).
 \label{Equ:Exa2Hypocycloid}
\end{equation}
In particular, if $h(r,\theta) > 0$, then $\A(f)$ has a bounded component of order $(3,3)$. 
For example, if we require that $|c| > 2.5$ then we obtain the following numerical 
intervals in the argument of $c$ which ensures the existence of a component of order $(3,3)$:
\[
 \arg(c) \in [0.08\pi, 0.42\pi] \ \cup \ [0.58\pi, 0.92\pi] \ \cup \ [-0.92\pi, -0.58\pi] \ \cup \ [-0.42\pi, -0.08\pi].
\]
Similarly, if $|c| > 3.5$ then we obtain the following numerical 
intervals:
\[
\arg(c) \in [ 0.01\pi, 0.49\pi] \ \cup \ [0.51\pi, 0.99\pi] \ \cup \  [-0.99\pi, -0.51\pi] \ \cup \ [-0.49\pi,-0.01\pi].
\]
\begin{figure}
\ifpictures
\includegraphics[width=0.3\linewidth]{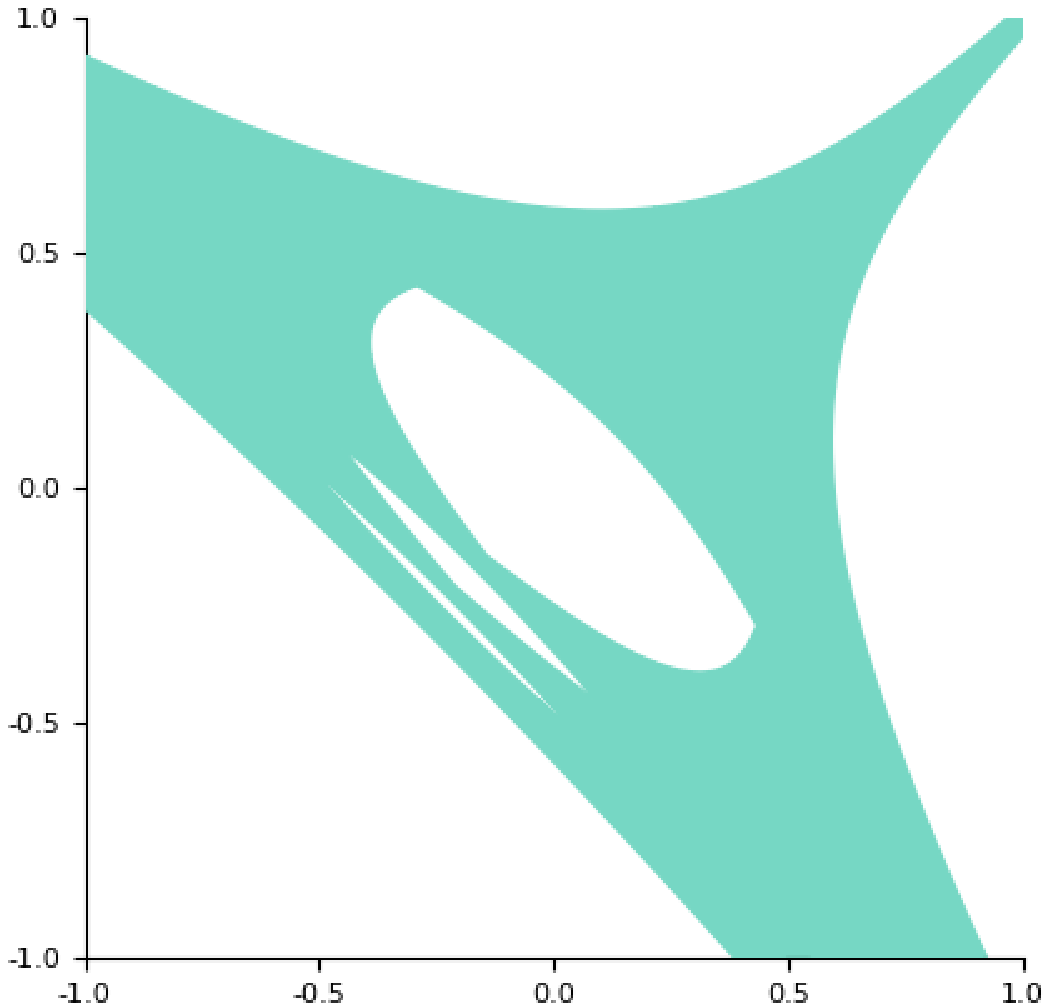} \quad
\includegraphics[width=0.3\linewidth]{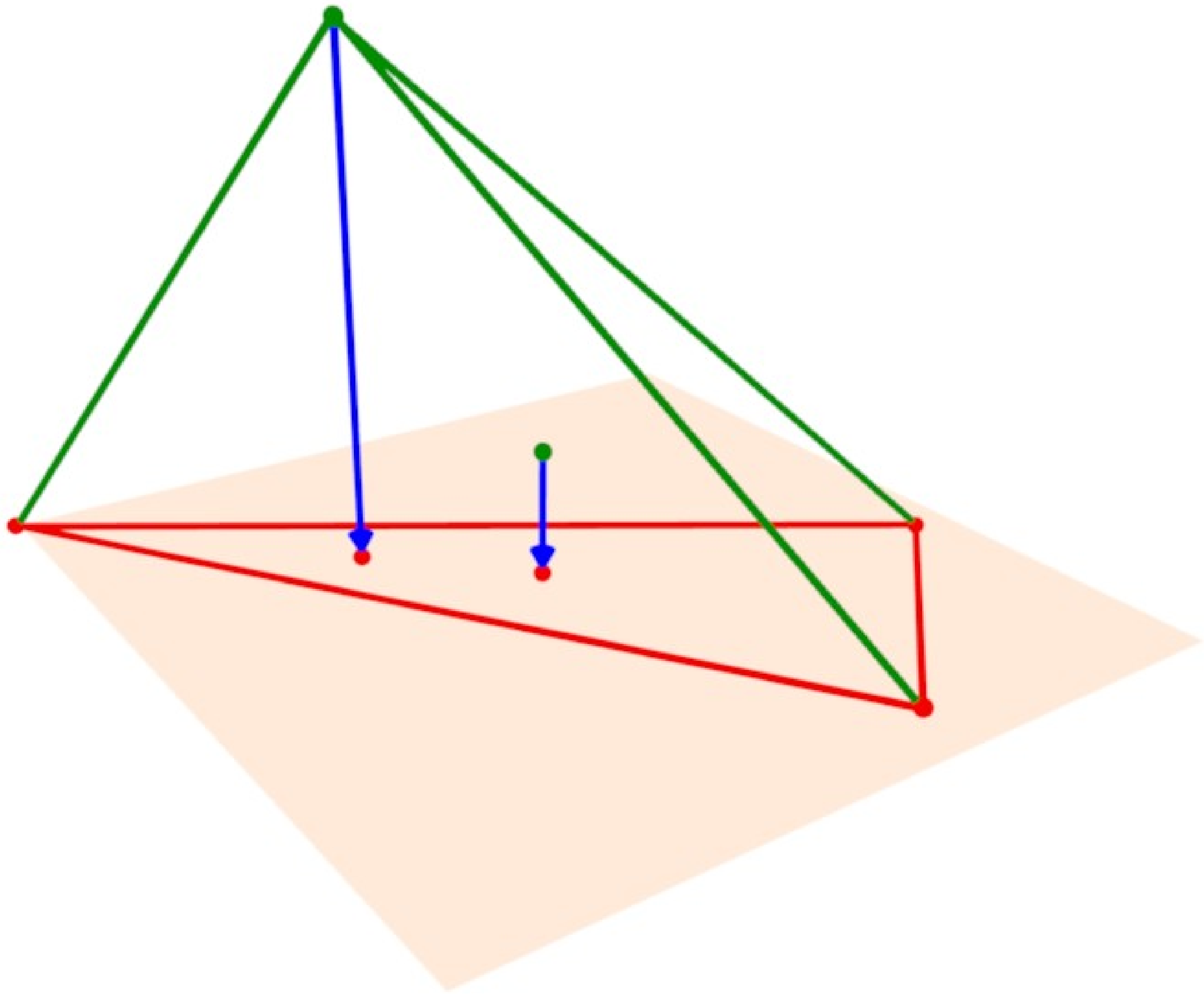} \quad
\includegraphics[width=0.3\linewidth]{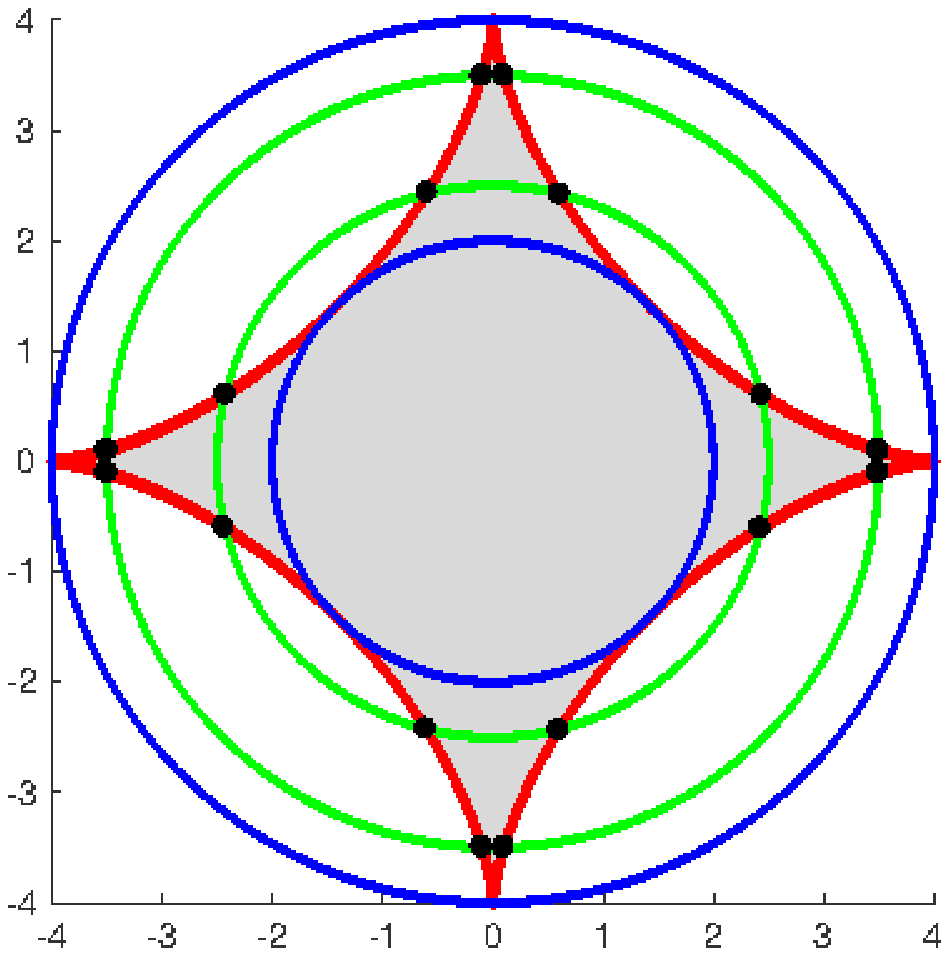}
\fi
\caption{Left, the amoeba of $f(w_1,w_2)$ from Example~\ref{ex:Sec6:Ex2} for $c = 3.5\,e^{0.63\, \pi i}$. Middle, the Newton polytope of $B$ with its projection onto $A$. 
Right, the region \eqref{Equ:CircuitPolynomialRegionSolid} bounded by the hypocycloid $h(r,\theta) = 0$ 
in red.}
\label{Fig:NewBounds2}
\end{figure}
\endexa
\end{example}

Corollary \ref{Cor:CircuitCertificate} is applicable also when $\rank(A) < \rank(B) + 1$. We leave the task of constructing an example for the curious reader, to keep the article at a reasonable length.

The fact that Theorem \ref{Thm:Theobald:deWolff:BarycentricCircuitCase} requires a barycentric circuit is a serious restriction. For non-barycentric circuits the only known certificates are certain upper and lower bounds on the norms of the coefficients, see \cite[Theorems 4.1 and 4.4]{TdW13}. Unfortunately, the lower bounds yields no new certificates in conjunction with Theorem \ref{Thm:AmoebaContainment}. The upper bounds are equivalent to
the lopsidedness criterion, see \cite[Theorem 5.3]{TdW13}.

We observed experimentally that a non-barycentric circuit polynomial with $n \geq 2$ has a solid amoeba if and only if the coefficient of the term whose exponent corresponds to the interior point is not contained in a region bounded by a hypocycloid determined by the remaining coefficients. We believe that a similar statement as Theorem~\ref{Thm:AmoebaContainment} holds for all circuits, this is the subject of an ongoing investigation.

\bibliographystyle{amsplain}

\end{document}